\documentclass{article}
\usepackage{graphicx} 
\usepackage{amsmath,amsthm}
\usepackage{color}
\newtheorem{theo}{Theorem}
\newtheorem{prop}[theo]{Proposition}
\newtheorem{lemm}[theo]{Lemma}
\newtheorem{conj}{Conjecture}
\newtheorem{clai}{Claim}[theo]

\newcommand{\cL}{\mathcal{L}}
\newcommand{\cA}{\mathcal{A}}
\newcommand{\cC}{\mathcal{C}}

\newcommand{\bet}{\mathcal{B}}
\newcommand{\lin}{\mathcal{L}}
\newcommand{\lco}{\ell\mathcal{C}}
\newcommand{\dia}{\textrm{diam}}

\title{Locally connected graphs: metric properties}
\author{Martín Matamala$^1$, Juan Pablo Peña$^1$, José Zamora$^2$\footnote{
Support by Basal program AFB170001 and Doctoral Fellowship grant 21211955, ANID, CHILE }
\\
  ($1$) DIM-CMM, CNRS-IRL 2807, Universidad de Chile, Chile\\
  ($2$) Depto. de Matemáticas, Universidad Andres Bello, Chile.
}
\date{August 2024}

\begin{document}

\maketitle
\begin{abstract}
In this work we show that any connected locally connected graph defines a metric space having at least as many lines as vertices with only three exception: the complete multipartite graphs $K_{1,2,2}$, $K_{2,2,2}$ and $K_{2,2,2,2}$. This proves that this class  fulfills a conjecture, proposed by Chen and Chvátal, saying that any metric space on $n$ points has at least $n$ lines or a line containing all the points. 
\end{abstract}
\section{Introduction}
A connected graph $G=(V,E)$ is locally connected if for each $v\in V$, the neighborhood of $v$ induces a connected graph. We denote by $\ell\mathcal{C}$ the class of all connected locally connected graphs. This class was introduced in \cite{CP1974}, where some interesting properties of these graphs were studied. Among others, any graph $G\in \lco$ is 2-connected and if it has maximum degree at most four, then it is either Hamiltonian or isomorphic to the complete multipartite graph $K_{1,1,3}$. Most combinatorial works about locally connected graphs have explored the occurrence of global properties under restrictions on the maximum degree (\cite{kikust1975,clark1981, hendry1989, gordon2011}). Concerning algorithmic properties, in \cite{kochol2005} it was proved that there is a linear time algorithm to decide if a locally connected graph has a proper 3-coloring.
The class of connected locally connected graphs contains two well-known classes of graphs: near triangulations and 2-connected chordal graphs. 

That a near triangulation is locally connected is a consequence of the following simple characterization of plane graphs being locally connected.

\begin{prop}\label{p:planelco}
A  plane graph  $G$ is locally connected if and only if every  vertex belongs to the boundary of at most one face whose boundary is not a triangle. 
\end{prop}
\begin{proof} Let $G$ be a plane graph. Assume that $G$ is locally connected. Let $v$ be a vertex and let $f$ be a face whose boundary $B$ contains $v$. If $B$ is not a triangle, then the two neighbors of $v$ in $B$ are connected by a path $P$ contained in $N(v)$ and such that $P+v$ is a cycle leaving $f$ in one of its faces. Then, any other face containing $v$ in its boundary has one edge of $P$ in its boundary, and hence its boundary is a triangle.

Conversely, let $v$ be a vertex and let $f$ be a face which is not a triangle containing $v$ in its boundary. By the assumption, all other boundaries containing $v$ are triangles. Hence, the neighborhood of $v$ induces a path in $G$. If the boundaries of all the faces containing $v$ are triangles, then the graph induced by its neighborhood is a cycle.
\end{proof}

To see that a 2-connected chordal graph $G$ is locally connected, let $u$ be a vertex in $G$ and  let $v$ and $w$ be two of its neighbors. Since $G$ is 2-connected, there is a path $P$ between $v$ and $w$ inside $g-u$. By choosing $P$ as an induced path, we get that every vertex in $P$ is a neighbor of $u$, as otherwise, $G$ contains an induced cycle of length at least four. Hence, the neighborhood of $u$ induces a connected subgraph of $G$.

The purpose of this work is to study the sets of lines defined by connected locally connected graphs. The \emph{line} defined in a graph $G=(V,E)$ by two distinct vertices $a$ and $b$ in $V$, is the set of all the vertices $c$ such that there is a shortest path containing $a, b$ and $c$. It is denoted by $\overline{ab}$. From this definition it is obvious that for three distinct vertices $a, b$ and $c$, $$c\in \overline{ab} \iff b\in \overline{ac}\iff a \in \overline{bc}.$$
To present our results it is useful to define the betweenness relation defined by $G$, denoted by $\mathcal{B}_G$ (first studied by Menger \cite{menger} and, more recently by Chv\'{a}tal \cite{chvatal2004}), as the set of all triples $(a,b,c)$, for $a, b$ and $c$ distinct vertices, such that a shortest path between $a$ and $c$ contains $b$. In terms of the distance function of $G$ defined by its shortest paths, denoted by $d_G$, we can define $\mathcal{B}_G$ as follows:
$$\mathcal{B}_G=\{(a,b,c)\in V^3\mid |\{a,b,c\}|=3 \textrm{ and } d_G(a,c)=d_G(a,b)+d_G(b,c)\}.$$
To ease the presentation we denote by $abc$ the triple $(a,b,c)$. Our main result in this work is the following theorem.
\begin{theo}\label{th:main}
    Let $G\in \ell\mathcal{C}$ with $n$ vertices. Then $G$ has less than $n$ lines if and only if it is isomorphic to one of the following complete multipartite graphs: $K_{1,2,2}, K_{2,2,2}$ or $K_{2,2,2,2}$.
\end{theo}
This implies that the metric space defined by $G\in \ell\mathcal{C}$ satisfies the so called Chen-Chvátal conjecture defined as follows \cite{chenchvatal}.
\begin{conj}\label{cj:cc}
Any metric space on $n$ points has a line containing the whole space (a universal line) or at least $n$ lines.     
\end{conj}

Chen and Chv\'{a}tal presented this conjecture as a version for metric spaces of a result known as the De-Bruijn-Erd\H{o}s theorem \cite{dbe}, which  states that if $E$ is a collection of subsets of a set $V$ such that the cardinality of each element in $E$ is at least two and at most $|V|-1$, and for each 2-set of $V$ there is exactly one set in $E$ containing it, then $\lvert E \rvert \geq n$. 
Chen and Chvátal proved that the number of lines in a metric space on $n$ points with no universal line is at $\Omega( \log n )$ \cite{chenchvatal}. This was later improved to $\Omega(n^{2/7})$ lines in \cite{chichv}, then to $\Omega(\sqrt{n})$ lines in \cite{aboulker} and, most recently to $\Omega(n^{2/3})$ lines in \cite{huang}. The latter remains the best asymptotic lower bound, up to date. Other results concerning the Chen-Chv\'{a}tal conjecture for specific types of metric spaces, mostly those defined by graphs, can be found in \cite{abbemaza, abkha, beboch, bekharo, ch2014, kantor, jualrove,montejano}.

In \cite{beboch} it was proved that chordal graphs satisfy Conjecture \ref{cj:cc}. The proof uses three ingredients.
\begin{enumerate} 
\item If a vertex $x$ has a unique neighbor $u$, then the line $\overline{xu}$ is universal. 
\item In a 2-connected chordal graph, if the lines $\overline{xu}$ and $\overline{xv}$ are equal, then the triple $uxv$ belongs to $\bet_G$. 
\item If for a simplicial vertex $x$ of a chordal graph we have that $\overline{xu}=\overline{xv}$, then the line $\overline{uv}$ is universal. 
\end{enumerate}
Hence, if a chordal graph has no universal line, then for any simplicial vertex $x$ the set of lines 
$$\{\overline{xu}\mid u\in V,u\neq x\}$$ has exactly $|V|-1$ lines
and $x$ has at least two adjacent neighbors $u$ and $v$. Since $x$ does not belong to the line $\overline{uv}$ one gets the conclusion that $G$ has at least $|V|$ lines.  Since in Theorem \ref{th:main} the existence of a universal line does not show its validity, the first and the third ingredients mentioned above are not meaningful.
However, in the proof of Theorem \ref{th:main} we shall need a generalization of the second ingredient, which we prove in Lemma \ref{l:zinthemiddle}.

It is known that 2-connected chordal graphs with $n$ vertices have a stronger property: they all have at least $n$ lines. This was proved as  
a step in the proof of the following theorem in \cite{abmaroza}.

    \begin{theo}[Theorem 2.1 in \cite{abmaroza}]\label{th:class}
Let $G$ be a graph with $n$ vertices such that every induced subgraph of $G$ either is a chordal graph, or has a cut-vertex, or has a nontrivial module.
If the numbers of lines in $G$ plus the number of bridges of $G$ is less than $n$, then $G$ is one of the following complete multipartite graphs 
$$K_{2,2},K_{2,3},K_{1,2,2},K_{2,2,2},K_{2,2,2,2}$$
or the graph $K'_{1,2,2}$, obtained from $K_{1,2,2}$ by deleting one edge incident to its vertex of degree four.
\end{theo}

Since  graphs $K_{2,2}, K_{2,3}$ and $K'_{1,2,2}$ are not locally connected and connected locally connected graphs are 2-connected, and then they do not have bridges, Theorem \ref{th:main} can be seen as a version of Theorem \ref{th:class} for connected locally connected graphs.

As a corollary, we get that any 2-connected plane graph with $n$ vertices, where no two nontriangular faces have a vertex in common has at least $n$ lines. In the best of the authors' knowledge it is the first class of 2-connected planar graphs for which Conjecture \ref{cj:cc} has been validated.

\section{Proof of Theorem \ref{th:main}}

The rest of the paper is devoted to the proof of Theorem \ref{th:main}. 
To this goal we consider two vertices $x$ and $y$ reaching the diameter of $G$ and 
the set of lines $\mathcal{L}^x$ and $\mathcal{L}^y$, where, for a vertex $z\in V$, 
$$\mathcal{L}^z=\{\overline{zu}\mid u\in V, u\neq z\}.$$

For  $z\in V$ we define a relation $R^z$ on $V\setminus \{z\}$ given by 
$$uR^zv \iff \overline{zu}=\overline{zv}.$$

It is clear that $R^z$ is an equivalence relation and that each equivalent class $[u]_z$ is associated to a unique line $\ell$ in $\cL^z$. We made the abuse of the notation $[\ell]_z$ to denote the equivalence class of those vertices $u$ such that $\overline{zu}=\ell$. When $z$ is clear from the context we just use $[\ ]$ instead of $[\ ]_z$.

We use the standard notation for the most common notions in Graph Theory. By instance, we denote by $N_G(v)$, the set of neighbors of the vertex $v$ and by $\dia(G)$, the diameter of the graph $G$. 

For distinct vertices $z,a,b\in V$, from the definition of a line, we have that $\overline{zb}=\overline{za}$  implies that one of the triples $azb,zab$ or $zab$ belongs to $\bet_G$. In the case of a graph $G\in \lco$ we now prove that there is only one possibility: $azb\in \bet_G$. This is a generalization of the second ingredient used in \cite{beboch} for 2-connected chordal graphs. 

\begin{lemm} \label{l:zinthemiddle}Let $G\in \lco$ and let $z,a,b$ three distinct vertices such that $[b]_z=[a]_z$. Then, $azb\in \bet_G$. 
\end{lemm}
\begin{proof}
    Without loss of generality, let us assume, for the sake of a contradiction, that $zab\in \bet_G$.
We shall prove that $a$ has a neighbor $a'$ such that $za'b\in \bet_G$ and $d_G(z,a)=d_G(z,a')$. This clearly is a contradiction, since $za'b\in \bet_G$ implies that $a'\in \overline{zb}$, while
$d_G(z,a)=d_G(z,a')$ and $d(a,a')=1$ implies that $a'\notin [\overline{za}]$.

To prove the existence of $a'$, we consider two neighbors of $a$, $z'$ and $b'$, such that $$d_G(z,z')+d_G(z',a)=d_G(z,a)$$ 
and 
$$d_G(a,b')+d_G(b',b)=d_G(a,b),$$
and a path $P$  between $z'$ and $b'$ ($z'$ may be $z$ and $b'$ may be $b$) included in the neighborhood of $a$. The path $P$ exists since $G\in \lco$.

On the one hand, for any two consecutive vertices $w,w'\in P$, we have that 
\begin{equation}\label{eq:cont}
|d_G(z,w)-d_G(z,w')|\leq 1
\end{equation}
On the other hand, for each $w\in P$, we have that 
$$d_G(z,z')=d_G(z,a)-1\leq d_G(z,w)\leq d_G(z,a)+1=d_G(z,b)-d_G(b',b).$$

Since $z',b'\in P$ with $d_G(z,z')=d_G(z,a)-1$ and $d_G(z,b')=d_G(z,a)+1$, from Inequality \ref{eq:cont}, we get that there is $w\in P$ such that $d_G(z,w)=d_G(z,a)$. 
Let $a'$ be the last vertex in $P$ (from $z'$ to $b'$) whose distance to $z$ is $d(z,a)$. Then, the neighbor $w$ of $a'$ in the subpath of $P$ from $a'$ to $b$ has distance to $z$ equal to $d_G(z,a)+1$. Since it is also a neighbor of $a$ we have that it belongs to $\overline{za}$ and thus to $\overline{zb}$. Hence, 
$d_G(z,b)=d_G(z,w)+d_G(w,b)$ which shows that 
\begin{eqnarray*}
d_G(a',b)&\leq &1+d_G(z,b)-d_G(z,w)\\
& = & d_G(z,b)-d_G(z,a)\\
& = &d_G(a,b).
\end{eqnarray*}
Since $d(z,a')=d(z,a)$ and $zab\in \bet_G$, we conclude that $za'b\in \bet_G$.
    
\end{proof}

\begin{lemm}\label{l:shortgen}
Let $a$ be such that $|[a]_z|\geq 2$ such that there is a vertex $v\in V$, $v\neq z,a$ with $zav\in \bet_G$. Then, there is a vertex $u\in N(a)$ such that $[a]_z\subseteq N(u)\cap N(z)$, and for each $b\in[a]_z$, $zbu\in \bet_G$.
\end{lemm}
\begin{proof} Let $v\in V$, $v\neq z,a$ with $zav\in \bet_G$. We take $v'$ (resp. $z'$) as the neighbor of $a$ in a shortest path between $a$ and $v$ (resp. $z$). Let $P$ be a path contained in the neighborhood of $a$ between $z'$ and $v'$. Since $G\in \lco$ such a path $P$ exists.  

Let $w$ be the last vertex in $P$ such that $d_G(z,w)=d_G(z,a)$ and let denote by $u$ the neighbor of $w$ in the subpath of $P$ between $w$ and $v'$.
Then, 
$$d_G(z,u)=d_G(z,a)+1=d_G(z,w)+1,$$
 which implies that $zwu\in \bet_G$ and that $zau\in \bet_G$, since $u$ is a neighbor of $a$.  

We prove that $zbu\in \bet_G$, for each $b\in [a]_z$. We know that $zau\in \bet_G$ which implies that $u\in \overline{za}$. Then, $u\in \overline{zb}$. Hence, if $zbu\notin \bet_G$, then we must have that $uzb$ or $zub$ belong to $\bet_G$. We prove that both cases lead to a contradiction.

\begin{enumerate}
\item $zub\in \bet_G$. By Lemma \ref{l:zinthemiddle} we get that $azb\in \bet_G$. When $azb$ and $zub$ belong to $\bet_G$ one gets that $azu$ also belongs to $\bet_G$, which is a contradiction with $zau\in \bet_G$.
\item $uzb\in \bet_G$. We already know that $uwz\in \bet_G$. Together with the fact that $uzb\in \bet_G$, we get that $wzb\in \bet_G$, which implies that $w\in \overline{zb}$. But, since $d_G(z,a)=d_G(z,w)$ and $d_G(a,w)=1$ we get that $w\notin \overline{za}$, which is a contradiction with the assumption $b\in [a]_z$.
\end{enumerate}

So far, we have proved that under the assumption that there is $v\in V$ such that $zav\in \bet_G$, then there is $u\in N(a)$ such that $zbu\in \bet_G$, for each $b\in [a]_z$, from which we get that 
$$
d_G(u,b)+d_G(b,z)=d_G(u,z)=d_G(z,a)+1,$$
for each $b\in [a]_z$.
By applying the same argument to $zbu\in \bet_G$, we get that there is a vertex $s\in N(b)$ such that 
$$
d_G(s,a)+d_G(a,z)=d_G(s,z)=d_G(z,b)+1.$$
By summing the last two equalities we get that 
$d_G(u,b)+d_G(s,a)=2$ which implies that $d_G(u,b)=d_G(s,a)=1$ and thus, $u\in N(b)\cap N(a)$ and $d_G(a,b)=2$. Hence, $[a]_z\subseteq  N(z)\cap N(u)$, as $azb\in \bet_G$.
\end{proof}
We now split the proof of Theorem \ref{th:main} in two parts. In the first part, we deal with graphs with diameter at least three; in the second part, we consider graphs of diameter two.

\begin{prop}\label{p:diametermorethantwo}
 Let $G$ be a connected locally connected graph of diameter at least three and $n$ vertices. Then, $G$ has at least $n$ lines.   
\end{prop}
\begin{proof}

\begin{clai}\label{cl:longgen}
Let $a$ be such that $|[a]_z|\geq 2$ and $[a]_z$ not included in $N(z)$. Then, for each $u,v\in [a]_z$ the line $\overline{uv}$ is given by 
$$\overline{uv}=\{w\in V\mid d_G(u,w)+d_G(w,v)=d_G(u,v)\},$$
it does contains $z$ and does not belong to $\lin^z$.
\end{clai}
\begin{proof}
    From Lemma \ref{l:zinthemiddle}
 we know that $uzv\in \bet_G$ and thus $z\in \overline{uv}$. Let $w\in V$. If $wvu\in \bet_G$, then $wuz\in \bet_G$, which is not possible from Lemma \ref{l:shortgen}. Similarly, the triple $vuw$ does not belong to $\bet_G$. Hence, $w\in \overline{uv}$ if and only if $uwv\in \bet_G$, proving the first statement. 

 Let us assume that $\overline{uv}=\overline{zw}$, for some $w\in V$, $w\neq z$.
 Without loss of generality let us assume that $d_G(u,w)\leq d_G(u,z)$. Then, $uzw\notin \bet_G$. Since $[a]_z$ is not included in $N(z)$, from Lemma \ref{l:shortgen}, we get that $zuw\notin \bet_G$. Thus, $u\in \overline{zw}$ if and only if $zwu\in \bet_G$. 
 Let $w'\in N(z)$ such that $zw'w\in \bet_G$ and let $v'\in N(z)$ such that $vv'z\in \bet_G$. Since $G\in \lco$, we know that there is a path $P$ in $N(z)$ between $v'$ and $w'$. Hence, there are $z', z''$ in $P$ such that $wzz'\notin \bet_G$ and $wzz''\in \bet_G$. As we did in the proof of Lemma \ref{l:zinthemiddle} we now show that $wz'v\in \bet_G$ and that $z'\notin \overline{wz}$.

From $wzz''\in \bet_G$ we get that $d_G(w,z'')=d_G(w,z)+1=d_G(w,v')$. Moreover,
$d_G(z',v)\leq d_G(z'',v)+1=d_G(z,v)$ and $d_G(w,z')\leq d_G(w,z)$, since $wzz'\notin \bet_G$ and $z'\in N(z)$. Therefore, 

$$d_G(w,v)\leq d_G(w,z')+d_G(z',v)\leq d_G(w,z)+d_G(z,v)=d_G(w,v),$$
where the last equality is due to $wzv\in \bet_G$. We conclude that 
$wz'v\in \bet_G$, which implies that $z'\in \overline{uv}$, and thus $d_G(w,z')=d_G(w,z)$. Hence, we get that $z'\notin \overline{xw}$, and thus the contradiction.
 
 \end{proof}

 Let $x$ and $y$ vertices such that $d_G(x,y)=diam(G)$.
By using previous results we can partition $\lin^x$ into the following sets.
$$\lin^x_0:=\{\ell\in \lin^x\mid |[\ell]_x|=1\},$$
$$\cC^x_1:=\{\ell\in \lin^x\mid |[\ell]_x|\geq 2 \textrm{ and } [\ell]_x\subseteq N(x)\}$$
and
$$\cC^x_2:=\{\ell\in \lin^x\mid |[\ell]_x|\geq 2 \textrm{ and } [\ell]_x\not \subseteq N(x)\}.$$

Let $V_1$ and $V_2$ the sets of vertices defined by 
$$V_1=\bigcup_{\ell \in \cC^x_1}[\ell]_x 
$$and
$$V_2=\bigcup_{\ell \in \cC^x_2}[\ell]_x\cap \{u\in V\mid d_G(x,u)\leq \dia(G)/2\}.
$$
We also define $\cC^y_{1,2}$ by 
$$\cC^y_{1,2}=\{\overline{yu}\mid u\in V_1\cup V_2\}.$$

\begin{clai}\label{cl:different}
For two distinct vertices $u,v\in V_1\cup V_2$ we have that $\overline{yu}\neq \overline{yv}.$  \end{clai} 
\begin{proof} Let us assume that $\overline{yu}=\overline{yv}.$
    From Lemma \ref{l:zinthemiddle} applied to $y$, we have that $uyv\in \bet_G$. If $u\in V_1$, then $d_G(x,u)=1$ and then $d_G(y,u)\geq \dia(G)-1$. Hence, $d_G(u,y)=\dia(G)-1$ and $d_G(y,v)=1$ which implies that $d_G(x,v)\geq \dia(G)-1$. As $\dia(G)\geq 3$ we get that $d_G(x,v)\geq \dia(G)/2>1$ which implies that $v\notin V_1\cup V_2$. 
    Therefore, we can assume that $u,v\in V_2$. In this case, we obtain the following contradiction:
    $$\dia(G)\geq d_G(u,v)=d_G(u,y)+d_G(y,v)>\dia(G),$$
since when $u,v\in V_2$ we get that $d_G(x,u),d_G(x,v)\leq \dia(G)/2$ and $xvy,xuy\notin \bet_G$, which implies that $d_G(y,u),d_G(y,v)>\dia(G)/2.$
\end{proof}

Therefore,
the set  ${\cC}^y_{1,2}$
    has exactly $|V_1|+|V_2|$ lines.

\begin{clai}\label{cl:disjoint} $\lin^x\cap \cC^y_{1,2}$ is empty.
\end{clai}
\begin{proof} 
Notice that for $v\in V_2$ we have that $xvy\notin \bet_G$, by Lemma \ref{l:shortgen}
, which implies that $x\notin \overline{yv}$, and thus $\overline{yv}\notin \lin^x$, when $v\in V_2$. 

Let $u\in V$ and $v\in V_1$ such that $\overline{xu}=\overline{yv}$.
Since $d_G(x,y)=\dia(G)$ we have that $xuy,xvy\in \bet_G$.
As $v\in V_1$, again from Lemma \ref{l:shortgen}, we have that $d_G(x,v)=1$ and that $[v]_x\subseteq N(x)$.
Then $x\in \overline{yv}$ if and only if $d_G(v,y)=\dia(G)-1$. This implies that $\overline{yv}\cap N(x)=\{v\}$.

As $\overline{xu}=\overline{yv}$ we have that $N(x)\cap \overline{xu}=\{v\}$ and hence,  $u=v$ or $xvu\in \bet_G$.
Since $[v]_x\subseteq \overline{xv}$ and $|[v]_x|\geq 2$, we get that $u\neq v$. But, when $xvu\in \bet_G$ we have that $u\in \overline{xv}$ which in turns implies that $[v]_x\subseteq \overline{xu}$ leading to a contradiction.
\end{proof}

\begin{clai}\label{cl:counting}
$G$ has at least $n$ lines.
\end{clai}
\begin{proof} From the definitions of $\lin^x$, $[\ell]_x$, $V_1$ and $V_2$, we have that 
$$|\lin^x|=|V|-1-\sum_{\ell\in \cC^x_1\cup \cC^x_2}\big(|[\ell]_x|-1\big),$$
$$|V_1|=\sum_{\ell\in \cC^x_1}|[\ell]_x|$$ and
$$|V_2|\geq \sum_{\ell\in \cC^x_2}\big(|[\ell]_x|-1\big).$$
The later inequality comes from the fact that at most one vertex in $[\ell]_x$ has $d_G(x,u)>\dia(G)/2$, by Lemma \ref{l:zinthemiddle}.

From Claim \ref{cl:disjoint} we have that 
\begin{align*}  
|\lin^x\cup \cC^y_{1,2}| & \geq |V|-1-
\sum_{\ell\in \cC^x_1\cup \cC^x_2}\big(|[\ell]_x|-1\big)+\sum_{\ell\in \cC^x_1}|[\ell]_x|+\sum_{\ell\in \cC^x_2}\big(|[\ell]_x|-1\big)\\
&=|V|-1+|\cC^x_1|.
\end{align*}

Hence, we can continue under the assumption that $\cC^x_1$ is empty.
If $\cC^x_2$ is empty we have that $|\lin^x|=|V|-1$.   Since $G\in \lco$, there is an edge $uv$, with $u,v\in N(x)$. Hence, $x\notin \overline{uv}$ and thus $\overline{uv}\notin \lin^x$ and we get the conclusion. 

If $\cC^x_2$ is non-empty, then for $\ell\in \cC^x_2$ and $u,v\in [\ell]_x$ we have that  $uxv\in \bet_G$, by Lemma \ref{l:zinthemiddle}, and $xuy,xvy\notin \bet_G$, by Lemma \ref{l:shortgen}. Hence, $uvy$, $yuv\notin \bet_G$. 

Moreover,
$$d_G(u,y)+d_G(y,v)+\dia(G)=d_G(u,y)+d_G(y,v)+d_G(x,u)+d_G(v,x)>2\dia(G)$$
which proves that $uyv\notin \bet_G$ either. Therefore, $y\notin \overline{uv}$
which implies that $\overline{uv}\notin \cC^y_{1,2}$.
To obtain the conclusion,  we use Claim \ref{cl:longgen}, which shows that $\overline{uv}\notin \lin^x.$
\end{proof}

\end{proof}

We now extend our result to connected locally connected graphs of diameter two. In this case we characterize those graphs in that class which have fewer lines than vertices.

\begin{prop}\label{p:diametertwo}
    Let $G$ be a connected locally connected graph of diameter two with $n$ vertices and fewer than $n$ lines. Then $G$ is isomorphic to $K_{1,2,2},K_{2,2,2}$ or $K_{2,2,2,2}$.
\end{prop}
\begin{proof} Let $N^2(x)$ be the set of all vertices at distance two from $x$. 
Let 
$$\cL^x_1=\{\overline{xu}\mid u\in N(x)\}$$ and 
$$\cL^x_2=\{\overline{xu}\mid u\in N^2(x)\}.$$ 
Then, $\cL^x=\cL^x_1\cup \cL^x_2$, as $V=N(x)\cup N^2(x)\cup \{x\}$, since $G$ has diameter two. This latter fact and Lemma \ref{l:zinthemiddle}
imply that $\cL^x_1\cap \cL^x_2$ is empty. Moreover, if $b\in [a]_x$, $b\neq a$, then $bxa\in \bet_G$ which shows that $a,b\in N(x)$, since $d_G(a,b)\leq 2$. Hence, $|\cL^x_2|=|N^2(x)|$. 

We also have that:
\begin{clai}\label{cl:atwounionathreenonempty}
$|\cL^x_1|<|N(x)|$. 
\end{clai}
\begin{proof} Since $|N^2(x)|=|\cL^x_2|$  and $\cL^x_1\cap \cL^x_2$ is empty, it is enough to prove that 
$$|\cL^x_1\cup\cL^x_2|<|N(x)|+|N^2(x)|.$$ 

Let us assume that $|\cL^x_1\cup\cL^x_2|=|N(x)|+|N^2(x)|=n-1$. Then,  we get that $\cL(G)=\cL^x_1\cup \cL^x_2$, since $|\cL(G)|\leq n-1$.

Since $G$ is locally connected, the neighbors of $x$ induce a connected subgraph of $G$. Hence, each neighbor $u$ of $x$ has a common neighbor $u'$ with $x$ which implies that the line $\ell=\overline{uu'}$ does not contain $x$. But this shows that $\ell$ does not belong to $\cL^x_1\cup \cL^x_2=\cL(G)$, a contradiction.
\end{proof}

\begin{clai}
    For each $u\in N(x)$, $[u]_x$ is an independent set which is also a module of $G$.
\end{clai}
\begin{proof}
By definition of the relation $R^x$, we know that for each $v\in [u]_x$,  $\overline{xu}=\overline{xv}$. Since $u,v\in N(x)$ this implies that $u$ and $v$ can not be adjacent. Therefore, each equivalence class of $R^x$ is an independent set.

 To see that $[u]_x$ is a module, we have to prove that $zu\in E \iff zu'\in E$, for each $u'\in [u]_x$.  If $z\in N(x)$, then $zu\in E$ if and only if $z\notin \overline{xu}$. This happens if and only if $z\notin \overline{xu'}$, for each $u'\in [u]_x$. In turns, this latter fact holds if and only if $zu'\in E$, for each $u'\in [u]_x$. Similarly, if $z\in N^2(x)$, $zu\in E$ if and only if $z\in \overline{xu}$ if and only if $zu'\in E$, for each $u'\in [u]_x$. 
\end{proof}

\begin{clai}\label{cl:desline}
    Let $\ell=\overline{vw}$ any line in $G$. Then, for each $u\in N(x)\setminus\{v,w\}$, $[u]_x\subseteq \ell$ or $[u]_x\cap \ell$ is empty. Moreover, for $v,w\in N(x)$, 
    when $vw\in E$, we have that $[v]_x\cup [w]_x\subseteq \ell$ and when $vw\notin E$ we have that  $([v]_x\cup [w]_x)\cap \ell=\{v,w\}$.
\end{clai}
\begin{proof} 
When $vw\in E$ we know that for each $u\in N(x)\setminus\{v,w\}$,  $u\in \ell$ if and only if $u\in (N(v)\cup N(w))\setminus (N(v)\cap N(w)).$ 

If $[u]_x\cap \ell$ is not empty we prove that $[u]_x\subseteq \ell$. Let $v\in [u]_x\cap \ell$. Since $[v]_x=[u]_x$ is a module, when $v\in \ell$  we get that $[v]_x=[u]_x\subseteq N(v)\setminus N(w)$ or $[v]_x=[u]_x\subseteq N(w)\setminus N(v)$. Hence, if $\ell\cap [u]_x$ is non-empty, then $[u]_x\subseteq \ell$.

For the second statement, since $vw\in E$, we have that $[v]_x\subseteq N(w)\setminus N(v)$  and $[w]_x\subseteq N(v)\setminus N(w)$,  because $[v]_x$ and $[w]_x$ are independent modules of $G$. Hence $[v]_x\cup [w]_x\subseteq \ell=\overline{vw}$.

When $vw\notin E$ we have that for each $u\in N(x)\setminus\{v,w\}$,  $u\in \ell$ implies that $u\in N(v)\cap N(w)$ which in turns implies that $[u]_x\subseteq N(v)\cap N(w)$ and thus $[u]_x\subseteq \ell$, since $[u]_x$ is a module of $G$. Since $vw\notin E$, we have $[v]_x\cup [w]_x$ is an independent set which implies that $([v]_x\cup [w]_x)\cap \ell=\{v,w\}$
\end{proof}
We partition $N(x)$ in the following three sets $A_i$, $i\in \{1,2,3\}$.
$$A_1=\{u\in N(x)\mid |[u]_x|=1\},$$
$$A_2=\{u\in N(x)\mid |[u]_x|=2\}$$
and
$$A_3=\{u\in N(x)\mid |[u]_x|\geq 3\}.$$
Then, 
$$|N(x)|=|A_1|+|A_2|+|A_3|.$$

From Claim \ref{cl:atwounionathreenonempty}
we know that $|N(x)|>|\cL^x_1|$ implying that $A_2\cup A_3$ is non-empty. Let $a_3$ be the number of lines $\overline{xu}$ in $\cL^x_1$ with $u\in A_3$. Then, the set $\cL^x_1$ contains $|A_1|+|A_2|/2+a_3$ lines.
Let 
$$A_2'=\{u\in A_2\mid \exists v\in N(x)\setminus [u]_x, uv\notin E\}$$ and 
$A''_2=A_2\setminus A_2'$.

We shall prove that the set $A''_2$ is non-empty. For this purpose, in the following claims, we define two sets of lines, $\cA_3$ and $\cA_2'$, having at least  $|A'_2|$ and $|A_3|$ lines, respectively, such that $\cL^x_1\cup \cL^x_2$, $\cA_3$ and $\cA'_2$ are pairwise disjoint.

\begin{clai}\label{cl:athree}
Let $\cA_3$ given by 
$$\cA_3=\bigcup_{u\in A_3}\cA_u,$$
where 
$$\cA_u=\{\overline{u'u''}\mid u',u''\in [u]_x\},$$
for $u\in A_3$. Then $\cA_3\cap (\cL^x_1\cup \cL^x_2)$ is empty and $|\cA_3|   \geq |A_3|$.

\end{clai}

\begin{proof}
    From Claim \ref{cl:desline} we know that each line $\overline{u'u''}$ in $\cA_u$ has exactly two vertices in $[u]_x$: $u',u''$. Thus, $|\cA_u|=\binom{|[u]_x|}{2}\geq |[u]_x|.$
From the same claim we have that each line in $\cL^x_1\cup \cL^x_2$ either contains $[v]_x$ or has not vertex from $[v]_x$, for each $v\in A_3$. Hence, $\cA_3\cap (\cL^x_1\cup \cL^x_2)$ is empty. 

Let $u,u',v\in A_3$ with $v\notin [u]_x$ and $u'\in [u]_x$. Then, $[v]_x\cap \overline{uu'}\in \{\emptyset, [v]_x\}$ which implies that $\cA_u\cap \cA_v$ is empty, and then $|\cA_3|\geq |A_3|$.
\end{proof}

\begin{clai}\label{cl:linetwoprime}
Let $\cA'_2$ be given by 
$$\cA'_2=\bigcup_{u\in A'_2}\bigcup_{v\in N(x)\setminus [u]_x, uv\notin E}\cA_{u,v},$$
where
$$\cA_{u,v}=\{\overline{st}\mid s\in [u]_x, t\in [v]_x\}.$$ 
Then, $\cA'_2\cap (\cL^x_1\cup \cL^x_2\cup \cA_3)$ is empty and $|\cA'_2|\geq |A'_2|.$
\end{clai}
\begin{proof}
From Claim \ref{cl:desline} we know that for each $u\in A_2'$ and $v\in N(x)\setminus [u]_x$ such that $uv\notin E$,  each line $\ell\in \cA_{u,v}$ satisfies 
$\ell\cap([u]_x\cup [v]_x)=\{u,v\}$ and $\ell\cap [w]_x\in \{\emptyset, [w]_x\}$, for each $w\in N(x)\setminus\{u,v\}$. Then, 
for $u'\in A'_2\setminus [u]_x, v'\in N(x)\setminus ([u']\cup [v]_x)$ the sets $\cA_{u,v}$ and $\cA_{u',v'}$ are disjoint.
The same claim shows that 
$\cA_{u,v}\cap (\cL^x\cup \cA_w)$ is empty, for each $w\in A_3$. In fact, a line in 
$\cA_{u,v}$ has exactly one vertex in $[u]_x$, the vertex $u$, and one vertex in $[v]_x$,  the vertex $v$, and $[u]_x\neq [v]_x$, while a line in $\cA_w$ contains exactly two vertices in $[w]_x$, and do not overlap with neither $[u]_x$ nor  $[v]_x$. 
We also know that $|\cA_{u,v}|=|[u]_x||[v]_x|\geq |[u]_x|=2$ from which we get that $\cA'_2$ 
contains at least $|A'_2|$ lines.
\end{proof}

\begin{clai}\label{cl:atwotwoprimesnonempty}
The set $A''_2$ is non-empty.    
\end{clai}

\begin{proof}
From previous claims we have that the set $\cL^x_1\cup \cL^x_2\cup \cA_3\cup \cA'_2$ contains
at least 
$$|N^2(x)|+|A_1|+|A_2|/2+a_3+|A_3|+|A'_2|=|N^2(x)|+|N(x)|+|A'_2|/2-|A''_2|/2+a_3$$
lines. Since $|\cL(G)|\leq n-1= |N(x)|+|N^2(x)|$ we get that $|A'_2|/2+a_3\leq |A''_2|/2$ which implies that $A''_2$ is non-empty as, from Claim \ref{cl:atwounionathreenonempty}, we know that $A_2\cup A_3$ is non-empty. 
\end{proof}

\begin{clai} If $N(x)\neq A''_2$, then $G$ is isomorphic to $K_{1,2,2}$. 
\end{clai}
\begin{proof}
    
By definition of $A''_2$ we know that, for each $u\in A''_2$ and $v\in N(x)\setminus [u]_x$ the edge $uv$ belongs to $E$. Then, for each $w\in N(x)\setminus ([u]_x\cup [v]_x)$ we have that $wu$ and $wv$ belong to $E$. From this and the fact that $[u]_x$ and $[v]_x$ are independent sets we get that 
$$\overline{uv}\cap N(x)=[u]_x\cup [v]_x.$$

Therefore, 
the set of lines 
$$\cA''_2=\{\overline{uv}\mid u,v\in A''_2, u\notin [v]_x\}$$
has exactly $\binom{|A''_2|/2}{2}$ lines and 
the set of lines 
$$\cA'''_2=\{\overline{uv}\mid u\in N(x)\setminus A''_2, v\in A_2''\}$$
has at least $|A''|/2$ lines, since $N(x)\setminus A''_2$ is non-empty. 
Moreover, $\cA''_2\cap \cA'''_2$ is empty. 

Therefore the set $$\cL^x\cup \cA_3\cup \cA'_2 \cup \cA''_2\cup \cA'''_2$$
contains at least 
\begin{eqnarray*}
    &  &|\cL^x_2|+|\cL^x_1|+|A_3|+|A_2'|+\binom{|A''_2|/2}{2}+|A''_2|/2\\
    &=& |\cL^x_2|+|A_1|+|A_2|/2+a_3+|A_3|+|A_2'|+\binom{|A''_2|/2}{2}+|A''_2|/2\\
    &=&n-1+a_3+|A'_2|/2+\binom{|A''_2|/2}{2},
\end{eqnarray*}
where the last equality comes from the fact that  $$n-1=|N^2(x)|+|N(x)|=|\cL^x_2|+|A_1|+|A_2'|+|A_2''|+|A_3|.$$

 Since we have that $n-1\geq |\cL(G)|$ and $A''_2$ is non-empty, we get that $a_3=|A'_2|=0$ and $|A''_2|=2$. Thus  $N(x)=A_1\cup\{u,v\}$, with $A''_2=\{u,v\}$ and 
 $$\cL(G)=\cL^x\cup \cA_3\cup \cA'_2 \cup \cA''_2\cup \cA'''_2=\cL^x\cup \cA'''_2.$$
 
 This shows that for each $b\in N^2(x)$ we have that  $bu,bv\in E$. Otherwise, the line $\overline{bu}$ does not contain $x$ and $\overline{bu}\cap [u]_x=\{u\} $ which proves that it does not belong to $\cL^x_1\cup \cL^x_2\cup \cA'''_2$. 
 
 From this we get that the line $\overline{uv}$ contains all the vertices, showing that it belongs to $\cL^x_1\cup \cL_2$. But no line in $\cL^x_1$ can contain $N(x)$, since $G$ is locally connected.  From this we get that $\overline{uv}$ belongs to $\cL^x_2$ and thus $N^2(x)=\{b\}$.

Assume that $A_1$ has at least two distinct vertices $u$ and $u'$. Then, $\overline{xu}\neq \overline{xu'}$ which implies, w.l.o.g., that $u$ has a neighbor $u''$ in $A_1$ which is not a neighbor of $u'$ (it could be $u'$). Hence, 
the line $\overline{uu''}$ does not intersect $A''_2\cup \{x\}$ which implies that it does not belong to $\cL^x\cup \cA'''_2.$ Thus, $A_1$ is a singleton.
 Since $G$ is locally connected we have that  the set of neighbors of $b$ induces a connected subgraph. Then, $b$ is adjacent to the vertex in $A_1$,  
 which shows that  $G$ is isomorphic to $K_{1,2,2}$. 
\end{proof}

Now consider the case when $N(x)=A''_2$. 
Thus, since $G$ is locally connected we have that $|A''_2|\geq 4$. 
In this case, the set $\cL^x_1\cup \cL^x_2\cup \cA''_2$ has $$|N^2(x)|+|A''_2|/2+\binom{|A''_2|/2}{2}=
n-1+\binom{|A''_2|/2}{2}-|A''_2|/2
$$ lines.
Thus, $\binom{|A''_2|/2}{2}-|A''_2|/2\leq 0$. This  implies that $|A''_2|\in \{4,6\}$ which shows that 
$\binom{|A''_2|/2}{2}-|A''_2|/2\in \{-1,0\}$.

As before, if for some $b\in N^2(x)$ and $u\in A''_2$ we have that $bu\notin E$, then the lines $\overline{ub}$ and 
$\overline{vb}$, for $[u]_x=\{u,v\}$, are different, they do not contains $x$ and contain exactly one vertex in $[u]_x$. Thus, 
they do not belong to $\cL^x_1\cup \cL^x_2\cup \cA''_2$ which is not possible since 
the set $\cL^x_1\cup \cL^x_2\cup \cA''_2$ has already at least $n-2$ lines.

Therefore, for each $b\in N^2(x)$ and $u\in A''_2$ we have that $ub\in E$ and we get that the line $\ell=\overline{uv}$, with $u\in A''_2$ and $v\in [u]_x$, contains every vertex of $G$. 

In order to conclude, we need to show that $N^2(x)=\{b\}$. For the sake of a contradiction, let us assume that there is $b'\in N^2(x)$, with $b'\neq b$. 
Then, the line $\ell$ does not belong to $\cL^x_1\cup \cL^x_2$ which implies that 
$$\cL(G)=\cL^x_1\cup\cL^x_2\cup \cA''_2\cup \{\ell\}.$$
But, the line $\overline{bb'}$ does not belong to $\cL^x_1\cup\cL^x_2\cup\{\ell\}$, since it does not contains $x$, and it does not belong to $\cA''_2$ either, as no line in $\cA''_2$ has a vertex in $N^2(x)$. This shows the contradiction $\overline{bb'}\notin \cL(G)$.

Therefore, $N^2(x)=\{b\}$ and then 
the graph $G$ is isomorphic to either $K_{2,2,2}$ or to $K_{2,2,2,2}$.

\end{proof}

\section{Conclusion}
One may wonder whether Proposition \ref{p:diametertwo} is valid for any 2-connected graph. We know that it is not true. For an integer $k\geq 3$, consider the graph $H_{2k}$ obtained from two complete graphs with $k$ vertices each, joined by a perfect matching. This graph has diameter two and $\binom{k}{2}+1$ lines. Then, the graphs $H_6$ and $H_8$ have less lines than vertices. The same is valid for the induced subgraph of $H_6$ with five vertices, $H_5$, usually called the \emph{house}. So far, these together with the graph appearing in Theorem \ref{th:class} are the only 2-connected graphs of diameter two with less lines than vertices known to us and we believe that there are no others.

Proposition \ref{p:diametermorethantwo} is also not valid for arbitrary 2-connected graphs of diameter at least three. Indeed, by deleting an edge from the perfect matching used in the definition of the graph $H_6$, we get a 2-connected graph $H'_6$ of diameter three with only four lines. The same holds true for the graph $H'_8$ obtained when applying previous operation to the graph $H_8$. In this latter case, the resulting graph has seven lines and eight vertices. A third example is the graph $H''_8$ obtained from a complete graph with four vertices joined by a perfect matching to two vertex disjoint edges. This graph is 2-connected, has eight vertices, diameter three and only seven lines.

\bibliographystyle{plain}
\bibliography{sample}

\end{document}